\documentclass{amsart}

\usepackage{amssymb}
\usepackage{graphicx}
\usepackage{pb-diagram,pb-xy} % Paul Burchard commutative diagram package
\usepackage[cmtip,arrow]{xy} % For arrows in commutative diagrams

\numberwithin{equation}{section}

\theoremstyle{plain}
\newtheorem{theorem}[equation]{Theorem}
\newtheorem{thm}[equation]{Theorem}
\newtheorem{lemma}[equation]{Lemma}

\newtheorem{corollary}[equation]{Corollary}
\newtheorem{proposition}[equation]{Proposition}
\newtheorem{prop}[equation]{Proposition}
\newtheorem*{thm:chainrule}{Theorem \ref{thm:chainrule}}

\theoremstyle{definition}
\newtheorem{definition}[equation]{Definition}
\newtheorem{convention}[equation]{Convention}
\newtheorem{remark}[equation]{Remark}
\newtheorem{remarks}[equation]{Remarks}
\newtheorem{example}[equation]{Example}
\newtheorem{examples}[equation]{Examples}

\newcommand{\homeq}{\simeq}                     % 'is homotopy equivalent to'
\newcommand{\smsh}{\wedge}                      % smash product
\newcommand{\wdge}{\vee}                        % wedge product

\newcommand{\Wdge}{\bigvee}                     % indexed wedge product

\DeclareMathOperator*{\hocolim}{hocolim}

\DeclareMathOperator*{\holim}{holim}

\newcommand{\creff}{\operatorname{cr} }
\DeclareMathOperator*{\hofib}{hofib}
\DeclareMathOperator*{\hocofib}{hocofib}

\newcommand{\spectra}{\mathsf{Spec}}            % spectra

\newcommand{\weq}{\; \tilde{\longrightarrow} \;}      % weak equivalence
\newcommand{\fib}{\twoheadrightarrow}           % fibration

\newcommand{\der}{\partial}                     % derivative of a functor

\newcommand{\ord}[1]{$#1$\textsuperscript{th}}

\begin{document}

\title[Chain rule for functors of spectra]{A chain rule for Goodwillie derivatives of functors from spectra to spectra}
\author{Michael Ching}
\subjclass[2000]{55P42,55P65}
\date{\today}

\begin{abstract}
We prove a chain rule for the Goodwillie calculus of functors from spectra to spectra. We show that the (higher) derivatives of a composite functor $FG$ at a base object $X$ are given by taking the composition product (in the sense of symmetric sequences) of the derivatives of $F$ at $G(X)$ with the derivatives of $G$ at $X$. We also consider the question of finding $P_n(FG)$, and give an explicit formula for this when $F$ is homogeneous.
\end{abstract}

\maketitle

\thispagestyle{empty}

\section*{Introduction}

In this paper we prove a version of the chain rule in the Goodwillie calculus of functors from spectra to spectra. Let $\spectra$ be a model for the stable homotopy category and let $F: \spectra \to \spectra$ be a homotopy functor (i.e. $F$ preserves stable weak equivalences). Fix a base object $X \in \spectra$ and let $\spectra/X$ denote the category of spectra over $X$. Then the methods of Goodwillie \cite{goodwillie:2003} can be used to construct a \emph{Taylor tower} for $F$, analogous to the Taylor series of a function of a real variable, based at the object $X$. This tower is a sequence of functors $\spectra/X \to \spectra$:
\[ F(Y) \to \dots \to P^X_nF(Y) \to P^X_{n-1}F(Y) \to \dots \to P^X_0F(Y) = F(X) \]
that, for each map $Y \to X$, interpolates between $F(Y)$ and $F(X)$.

Again following Goodwillie, the \emph{layers} of the Taylor tower are the homotopy fibres
\[ D^X_nF := \hofib(P^X_nF \to P^X_{n-1}F). \]
There is a formula for these layers of the form
\[ D^X_nF(Y) \homeq (\der_nF(X) \smsh \hofib(Y \to X)^{\smsh n})_{h\Sigma_n}. \]
This formula holds when $Y$ is a finite cell object in $\spectra$ and for all $Y \to X$ if $F$ preserves filtered homotopy colimits. The object $\der_nF(X)$ is a spectrum with an action of the symmetric group $\Sigma_n$ and we refer to it as the \emph{\ord{n} derivative of $F$ at $X$}. (In the analogy with ordinary calculus, it plays the role of the \ord{n} derivative of $F$ evaluated at $X$.)

The derivatives of $F$ at $X$ together form a \emph{symmetric sequence} in $\spectra$, that is a sequence of objects with actions of the symmetric groups. We write $\der_*F(X)$ for this symmetric sequence. Recall that the category of symmetric sequences supports a monoidal product, called the \emph{composition product}, written $\circ$. (See Definition \ref{def:compprod}.)

The main result of this paper is the following.
\begin{thm:chainrule}
Let $F,G: \spectra \to \spectra$ be homotopy functors and suppose that $F$ preserves filtered homotopy colimits. Then we have the following formula:
\[ \der_*(FG)(X) \homeq \der_*F(GX) \circ \der_*G(X). \]
\end{thm:chainrule}
We call this a `chain rule' since it expresses the derivatives of a composite of functors in terms of the derivatives of the individual functors. It bears a striking similarity to the corresponding result for functions of real-variables \cite{johnson:2002}. Note that the condition that $F$ preserves filtered homotopy colimits is essential. See \ref{ex:counterexample} for a counterexample in the case that $F$ does not have this property.

The derivatives of a functor determine the layers in the Taylor tower, but there are still extension problems in recovering the whole tower. We therefore consider also the problem of expressing $P^X_n(FG)(Y)$ in terms of the Taylor towers of $F$ and $G$. Our method gives an answer to this question in the case that the map $Y \to X$ has a section, but there is no nice general formula and it is not clear that the answer can be used in practice to calculate Taylor towers. We however give explicit descriptions of $P_2(FG)$ and $P_3(FG)$ where the calculations are not so hard. Along the way, we will consider the case where $F$ is homogeneous for which there \emph{is} a simple formula for $P_n(FG)$. (See Theorem \ref{thm:hom}.)

We should point out that our proof is non-constructive in the sense that we do not define specific models for the derivatives of a functor that satisfy the equivalence of Theorem \ref{thm:chainrule}. This has the downside that we cannot use our proof to obtain explicit operad or cooperad structures on the derivatives of monads and comonads. In separate work between the author and Greg Arone, an alternative approach to the Theorem is being developed that \emph{does} give specific models for these derivatives and produces interesting operad and cooperad structures on the derivatives of certain functors. The current paper, on the other hand, has the advantage that its methods reveal something about the chain rule on the level of Taylor towers and not just for derivatives.

Finally, we remark that while we write out the proof for functors of spectra, many of the underlying ideas apply equally well to functors between any stable model categories. The statement of Theorem \ref{thm:chainrule} would need to be adjusted in this setting to take into account the more complicated nature of the derivatives of a functor in general. Many of the intermediate results, however, carry over directly. See Remark \ref{rem:stable} for more details.

\subsection*{Outline of the paper}
In \S\ref{sec:prelims}, we define precisely what we mean by the derivatives of a functor, we define the composition product of symmetric sequences, and state the main Theorem. The proof of this theorem occupies sections 2-4: in \S2 we describe a map that on \ord{n} derivatives will give the equivalence of the theorem; in \S3, we study the Taylor tower of $FG$ when $F$ is homogeneous, and in \S4, we complete the proof for all $F$. Section 3 includes our explicit description of $P_n(FG)$ when $F$ is homogeneous. In \S5, we look at the question of the full Taylor tower of $FG$ (for general $F$) and not just the derivatives, and give formulas for the lowest few terms in the tower. Finally, in \S6, we prove some fundamental results about Taylor towers of composite functors that make the rest of the paper possible. These results say that to calculate $P_n(FG)$ it is sufficient to know $P_nF$ and $P_nG$ (if $F$ is finitary).

\subsection*{Acknowledgements}
This work owes a considerable debt to conversations between the author and Greg Arone, and is related to a joint project to prove chain rules for derivatives in a much wider context. More specifically, one of the key ideas of this paper, namely that we can use the diagonal maps and co-cross-effects of McCarthy to study the chain rule, was suggested by Arone.

The author would also like to thank Tom Goodwillie for useful conversations related to the proof of Proposition \ref{prop:fundamental}.

\section{Preliminaries} \label{sec:prelims}
In this section, we describe the Taylor tower, and derivatives, for functors from spectra to spectra. We then recall the composition product for symmetric sequences and state our main theorem.

\begin{definition} \label{def:spectra}
Let $\spectra$ be one of the usual models for the stable homotopy category. Specifically we require that $\spectra$ be a symmetric monoidal proper simplicial cofibrantly generated model category. Any of the standard examples can be used, but in the proof of Proposition \ref{prop:fundamental} it will be convenient to take $\spectra$ to be the category of symmetric spectra based on simplicial sets, as in \cite{hovey/shipley/smith:2000}. We shall refer to the objects of $\spectra$ as \emph{spectra}.

For a fixed spectrum $X$, we write $\spectra/X$ for the category of spectra over $X$, that is whose objects are maps $Y \to X$ and $\spectra_X$ for the category of spectra over and under $X$, that is, the category whose objects are sequences $X \to Y \to X$ that compose to the identity on $X$.
\end{definition}

\begin{remark} \label{rem:stable}
The symmetric monoidal structure on $\spectra$, and the fact that it has the `usual' stable homotopy category are used only to express our chain rule in terms of derivatives. The construction of the map $\Delta$ (Definition \ref{def:delta} below), the result that it is a $D_n$-equivalence (Proposition \ref{prop:equiv}), and hence the results of \S\ref{sec:tower} on the Taylor tower of $FG$, require only that $\spectra$ be a (proper simplicial cofibrantly generated) \emph{stable} model category.

A stable model category is one in which the suspension functor is an equivalence on the homotopy category. From our point of view, the key property enjoyed by a stable model category is that homotopy pushout and homotopy pullback squares coincide. In particular, fibre and cofibre sequences coincide, and so linear functors preserve fibre sequences.
\end{remark}

\begin{definition} \label{def:functors}
We are interested in studying functors $F: \spectra \to \spectra$, but we shall need to impose various conditions. We make the following definitions:
\begin{itemize}
  \item $F$ is a \emph{homotopy functor} if it preserves weak equivalences;
  \item $F$ is \emph{simplicial} if it preserves the simplicial enrichment in $\spectra$, that is, if it induces maps of simplicial mapping objects;
  \item $F$ is \emph{finitary} if it preserves filtered homotopy colimits;
  \item $F$ is \emph{reduced} if $F(*) \homeq *$.
\end{itemize}
\end{definition}

\begin{convention} \label{convention}
All functors in this paper are assumed to be simplicial homotopy functors. The conditions of being finitary or reduced will be stated explicitly when needed. We shall also at times need functors of several variables, i.e. of the form $H:\spectra^k \to \spectra$. These are assumed to be simplicial homotopy functors with respect to each variable.
\end{convention}

\begin{definition}[Taylor tower for functors from spectra to spectra]
Goodwillie explicitly describes calculus for functors of topological spaces, but his ideas extend easily to a more general setting, for example, see Kuhn \cite{kuhn:2007}. In particular, a simplicial homotopy functor $F:\spectra \to \spectra$ has, for each $X \in \spectra$, a \emph{Taylor tower at $X$}, that is a sequence of functors
\[ F \to \dots \to P^X_nF \to P^X_{n-1}F \to \dots P^X_0F = F(X) \]
in which $P^X_nF$ is an \emph{$n$-excisive} functor $\spectra/X \to \spectra$, i.e. takes strongly homotopy cocartesian $(n+1)$-cubes in $\spectra/X$ to homotopy cartesian cubes in $\spectra$. (See \cite{goodwillie:1991}.) The \emph{\ord{n} layer} in the tower is the homotopy fibre
\[ D^X_nF := \hofib(P^X_nF \to P^X_{n-1}F). \]
and this is \emph{$n$-homogeneous}. We recall the precise formulation of the $P_n$ construction in Definition \ref{def:T_n}, and we refer to \cite[\S1]{goodwillie:2003} for the following properties satisfied by $P_n$ and $D_n$ which we shall use repeatedly:
\begin{itemize}
  \item $P_n$ and $D_n$ commute with finite homotopy limits (including homotopy fibres) and all homotopy colimits (since we are dealing with spectrum-valued functors).
\end{itemize}

We shall refer to $X$ as the \emph{base object} for the Taylor tower (or for the layers, or derivatives). It is the fixed point about which the Taylor expansion is taking place. When the base object is the trivial spectrum $*$, we write simply $P_nF$ and $D_nF$.
\end{definition}

\begin{definition}[Derivatives at $*$] \label{def:derivative-*}
The derivatives of a functor are the spectra that classify the homogeneous layers in the Taylor tower. Derivatives at base object $*$ work in the same way as for functors of spaces. Specifically, given $F: \spectra \to \spectra$, there is a spectrum $\der_nF$ with $\Sigma_n$-action such that
\[ D_nF(Y) \homeq (\der_nF \smsh Y^{\smsh n})_{h\Sigma_n} \]
for finite cell spectra $Y$ and for all $Y \in \spectra$ if $F$ is finitary. The spectra $\der_nF$ are the \emph{derivatives of $F$ at $*$}.

We can give a direct definition of $\der_nF$ as
\[ \der_nF := \creff_n(D_nF)(S,\dots,S) \]
where $\creff_n$ is the \ord{n} cross-effect construction (see \cite[\S3]{goodwillie:2003}), and $S$ denotes the sphere spectrum.
\end{definition}

\begin{definition}[Derivatives at a general $X$] \label{def:derivative}
The case of functors from spectra to spectra is not explicitly covered by \cite{goodwillie:2003} and for derivatives at a base object other than $*$, the situation is not quite the same as for functors of spaces. We make the following definition.

Let $F: \spectra \to \spectra$ be a homotopy functor. The \emph{\ord{n} derivative of $F$ at $X$} is defined to be the \ord{n} derivative at $*$ of the functor
\[ F_X(Z) = \hofib(F(Z \wdge X) \to F(X)) \]
That is, we set
\[ \der_nF(X) := \der_n(F_X). \]
Taking $X = *$, note that $F_*$ is the same as the `reduction' (or first cross-effect) of the functor $F$ and has equivalent derivatives, so that this definition of the derivative at $*$ agrees up to equivalence with that of Definition \ref{def:derivative-*}.
\end{definition}

\begin{remark} \label{rem:spaces}
In the case of functors of topological spaces, Goodwillie shows that the derivative at a base space $X \neq *$ can be thought of as a family of spectra parametrized over $X$ (see \cite[5.6]{goodwillie:2003}). In the case of spectra, we will see in the next proposition that the single spectrum $\der_nF(X)$ is enough to classify the homogeneous layers. This reflects the fact that the category of spectra over and under a fixed spectrum is Quillen equivalent again to $\spectra$ and is another consequence of the stability of the model category $\spectra$.
\end{remark}

\begin{proposition}[Classification of the layers] \label{prop:d_nF}
Let $F: \spectra \to \spectra$ be a homotopy functor and let $Y \to X$ be a map of spectra. Then
\[ D^X_nF(Y) \homeq (\der_nF(X) \smsh \hofib(Y \to X)^{\smsh n})_{h\Sigma_n}. \]
This holds whenever $\hofib(Y \to X)$ is equivalent to a finite cell spectrum, and for all $Y \to X$ if $F$ is finitary.
\end{proposition}

The following lemma will be useful in proving this proposition:

\begin{lemma} \label{lem:hofib}
The functor $\spectra/X \to \spectra$ given by
\[ (Y \to X) \mapsto \hofib(Y \to X) \]
preserves homotopy cocartesian squares.
\end{lemma}
\begin{proof}
Suppose we have a commutative square of spectra
\[ \begin{diagram}
  \node{A} \arrow{e} \arrow{s} \node{B} \arrow{s} \\
  \node{B'} \arrow{e} \node{C}
\end{diagram} \]
together with a map $C \to X$ (which makes all the spectra in this diagram into objects in $\spectra/X$). The forgetful functor $\spectra/X \to \spectra$ reflects homotopy colimits, and homotopy cartesian and cocartesian are the same for spectra, so it is enough to show that if the above square is homotopy cartesian, then so is the corresponding square of homotopy fibres. This follows from the commutativity of homotopy limits.
\end{proof}

\begin{proof}[Proof of Proposition \ref{prop:d_nF}]
If $F$ is finitary, then so is $F_X$, so the right-hand side of the claimed equivalence is the same as
\[ D_n(F_X)(\hofib(Y \to X)), \]
by definition of $\der_nF(X)$ and the usual derivative formula for $D_n(F_X)$.

By Lemma \ref{lem:hofib}, the functor $Y \mapsto \hofib(Y \to X)$ commutes with the Taylor tower constructions and so we have
\[ D_n(F_X)(\hofib(Y \to X)) \homeq D^X_n(Y \mapsto F_X(\hofib(Y \to X))) \]
which by the definition of $F_X$ is the same as
\[ \tag{*} D^X_n(Y \mapsto \hofib(F(\hofib(Y \to X) \wdge X) \to FX)). \]
Now Goodwillie shows in \cite[4.1]{goodwillie:2003} that two homogeneous functors on $\spectra/X$ are equivalent if and only if they are equivalent when pre-composed with the forgetful functor $\spectra_X \to \spectra/X$, that is, if and only if they are equivalent when applied to maps $Y \to X$ that have a chosen section. (To be precise, Goodwillie proves this in the case of functors of topological spaces, but the same proof applies to functors of spectra.)

If the map $Y \to X$ has a section, then we have an equivalence
\[ \hofib(Y \to X) \wdge X \weq Y \]
and so
\[ \hofib(F(\hofib(Y \to X) \wdge X) \to FX) \homeq \hofib(FY \to FX). \]
It follows that (*) is equivalent to
\[ D^X_n(Y \mapsto \hofib(FY \to FX)). \]
But $D^X_n$ commutes with homotopy fibres and is trivial when applied to the constant functor $Y \mapsto FX$. We therefore get
\[ D^X_n(Y \mapsto \hofib(FY \to FX)) \homeq D^X_nF \]
which completes the proof of the Proposition.
\end{proof}

\begin{example} \label{ex:identity}
The identity functor on topological spaces has an interesting calculus, but this is not the case for the identity on $\spectra$. Since homotopy cartesian and cocartesian squares coincide for spectra, it follows that the identity is linear and we have
\[ P^X_1I(Y) = Y, \quad D^X_1I(Y) = \hofib(Y \to X)\]
and hence
\[ \der_nI(X) \homeq \begin{cases} S, & \text{if $n = 1$}; \\ *, & \text{otherwise}. \end{cases} \]
\end{example}

We now introduce the composition product of symmetric sequences, making sure we are precise about the way partitions are handled.

\begin{definition}[Partitions]
For a positive integer $n$, let $P(n)$ denote the set of unordered partitions of $n$. An element of $P(n)$ can then be written uniquely as a nondecreasing sequence of positive integers whose sum is $n$. Thus, for example, $P(3) = \{(3),(1,2),(1,1,1)\}$, and so on. We can also specify an element $\lambda$ of $P(n)$ uniquely by writing
\[ n = k_1 l_1 + \dots + k_r l_r \]
for some $k_i \geq 1$ and $1 \leq l_1 < \dots < l_r$. The integers $l_i$ are the numbers that appear in the partition $\lambda$ and $k_i$ is the number of times that $l_i$ appears.

For such a partition $\lambda \in P(n)$ we write
\[ H(\lambda) = (\Sigma_{l_1} \wr \Sigma_{k_1}) \times \dots \times (\Sigma_{l_r} \wr \Sigma_{k_r}) \leq \Sigma_n. \]
Here the wreath product $\Sigma_{l_i} \wr \Sigma_{k_i}$ is the subgroup of $\Sigma_{k_i l_i}$ generated by permutations of $k_i$ blocks of size $l_i$, and by the permutations within each block. We identify their product with a subgroup of $\Sigma_n$ using the decomposition $n = k_1 l_1 + \dots + k_r l_r$.

We will not use this fact explicitly, but note that the subgroup $H(\lambda)$ is the stabilizer of the transitive action of $\Sigma_n$ on the set of partitions of a set of $n$ elements that are of `type' $\lambda$. The cosets of $H(\lambda)$ thus correspond bijectively with such partitions.
\end{definition}

\begin{definition}(Composition product of symmetric sequences) \label{def:compprod}
A \emph{symmetric sequence} in $\spectra$ is a sequence of spectra $A = (A_1,A_2,\dots)$ with an action of $\Sigma_n$ on $A_n$. Thus the collection of all derivatives of a functor $F$ at a fixed $X$ forms a symmetric sequence $\der_*F(X)$.

Let $A,B$ be two symmetric sequences in $\spectra$. The composition product $A \circ B$ is the symmetric sequence given by
\[ (A \circ B)_n := \Wdge_{\lambda \in P(n)} (\Sigma_n)_+ \smsh_{H(\lambda)} A_k \smsh B_{l_1}^{\smsh k_1} \smsh \dots \smsh B_{l_r}^{\smsh k_r}. \]
Here $H(\lambda)$ acts freely on $\Sigma_n$ by (right) multiplication and on the right-hand side in the following way:
\begin{itemize}
  \item an element of $H(\lambda)$ determines an element of $\Sigma_{k_1} \times \dots \times \Sigma_{k_r}$ by combining the permutations of the various blocks in the wreath products, and hence an element of $\Sigma_k$ where $k = k_1 + \dots + k_r$;
  \item this element of $\Sigma_k$ then acts on the $A_k$ term via the given symmetric sequence structure of $A$;
  \item the element of $\Sigma_{k_1} \times \dots \times \Sigma_{k_r}$ also acts on
  \[ B_{l_1}^{\smsh k_1} \smsh \dots \smsh B_{l_r}^{\smsh k_r} \]
  by permuting the factors;
  \item the given element of $H(\lambda)$ also includes permutations within each block. These are elements of $\Sigma_{l_i}$ and act on the appropriate $B_{l_i}$ via the symmetric sequence structure on $B$;
  \item the overall action of $H(\lambda)$ on $A_k \smsh B_{l_1}^{\smsh k_1} \smsh \dots \smsh B_{l_r}^{\smsh k_r}$ is the combination of each of these individual actions.
\end{itemize}
The $\Sigma_n$-action on $(A \circ B)_n$ is determined by the left regular action on the $(\Sigma_n)_+$ factor in each term of the coproduct.
\end{definition}

\begin{remark}
Our definition is isomorphic to the traditional way of defining the composition product as a coproduct over \emph{ordered} partitions of $n$. The decomposition in terms of unordered partitions reflects our approach to proving the chain rule. As we will see, that approach is essentially to decompose $FG$ into pieces indexed by the unordered partitions of $n$. The derivatives of these pieces will then correspond to the terms of the coproduct in Definition \ref{def:compprod}.
\end{remark}

We are now in position to state the main result of this paper.

\begin{thm}[Chain rule for functors from spectra to spectra] \label{thm:chainrule}
Let $F,G: \spectra \to \spectra$ be homotopy functors and suppose that $F$ preserves filtered homotopy colimits. Then
\[ \der_*(FG)(X) \homeq \der_*F(GX) \circ \der_*G(X). \]
\end{thm}

\begin{example}
We can quickly verify Theorem \ref{thm:chainrule} in the case that either $F$ or $G$ is the identity functor $I$ on $\spectra$. In Example \ref{ex:identity} we showed that the derivatives of $I$ (at an arbitrary base object) form the unit for the monoidal product $\circ$. We therefore deduce that
\[ \der_*(FI)(X) \homeq \der_*F(X) \circ \der_*I(X)\]
and
\[ \der_*(IG)(X) \homeq \der_*I(GX) \circ \der_*G(X). \]
\end{example}

\begin{remark} \label{rem:firstderivatives}
For first derivatives, our result reduces to
\[ \der_1(FG)(X) \homeq \der_1F(GX) \smsh \der_1G(X) \]
which has a similar form to the chain rule formula of Klein and Rognes \cite{klein/rognes:2002}. Their result concerned functors of topological spaces in which case (as we noted in Remark \ref{rem:spaces}) the concept of derivative at a general base object is more complicated. In the case of spectra, we get this simpler formula.
\end{remark}

The following example from Kuhn \cite[7.4]{kuhn:2007} shows that the condition that $F$ be finitary in Theorem \ref{thm:chainrule} is essential.

\begin{example} \label{ex:counterexample}
Let $F(X) = L_E(X)$ be the Bousfield localization with respect to a generalized homology theory $E$, and let $G(X) = (X \smsh X)_{h\Sigma_2}$. As in Remark \ref{rem:firstderivatives}, Theorem \ref{thm:chainrule} says that for first derivatives (at $*$) we would expect
\[ \der_1(FG) \homeq \der_1(F) \smsh \der_1(G). \]
Now $\der_1(G) \homeq *$, so the right-hand side of this is trivial. But $\der_1(FG) \homeq P_1(FG)(S)$ and Kuhn points out that
\[ P_1(FG)(S) \homeq \hocofib(L_E(S) \smsh \mathbb{R}P^\infty \to L_E(\mathbb{R}P^\infty)). \]
If $L_E$ is a non-smashing localization, this need not be trivial. For example, if $E$ is mod $2$ K-theory, then $P_1(FG)(S) \neq *$. In this case then Theorem \ref{thm:chainrule} does not hold. Note that if $F$ were finitary, then $L_E$ would be a smashing localization and $P_1(FG)(S)$ would indeed be trivial.
\end{example}

One way to understand where the finitary condition comes in is via the following lemma, which in some sense underpins the entire paper. It will be used later to prove the important results Corollary \ref{cor:p_n(FG)} and Proposition \ref{prop:fundamental}. Notice that it is precisely this lemma that fails to hold in Example \ref{ex:counterexample}.

\begin{lemma} \label{lem:reduced}
Let $F:\spectra \to \spectra$ be a \emph{finitary} linear homotopy functor, and let $G:\spectra \to \spectra$ be an $r$-reduced homotopy functor (that is $P_{r-1}G \homeq *$). Then the composite $FG$ is also $r$-reduced (that is, $P_{r-1}(FG) \homeq *$).
\end{lemma}
\begin{proof}
The key is to show that
\[ P_n(FG) \homeq FP_nG \]
for all $n$. The result will then follow by setting $n = r-1$.

First we show that a linear functor of spectra preserves homotopy cartesian cubes of all dimensions. This is true for any linear functor and does not require the finitary hypothesis. We prove it by induction on the dimension $n$ of the cube. A homotopy cartesian $1$-cube is just a map that is an equivalence, so the base case $n = 1$ is true because $F$ is a homotopy functor.

A homotopy cartesian $n$-cube $\mathcal{X}$ can be viewed as a map $\mathcal{X}_0 \to \mathcal{X}_1$ between two $(n-1)$-cubes. Let $\mathcal{Y}$ be the termwise homotopy fibre of this map. Then $\mathcal{Y}$ is a homotopy cartesian $(n-1)$-cube. Since the linear functor $F$ preserves fibre sequences, it follows that $F(\mathcal{Y})$ is the homotopy fibre of the map
\[ F(\mathcal{X}_0) \to F(\mathcal{X}_1) \]
which can also be thought of as the $n$-cube $F(\mathcal{X})$. But by the induction hypothesis, $F(\mathcal{Y})$ is homotopy cartesian, from which it follows that $F(\mathcal{X})$ is also homotopy cartesian. This completes the induction.

It follows from this that $F$ preserves the homotopy limits that appear in the construction of $T_n$ (see Definition \ref{def:T_n}) and so we have
\[ F(T_n^kG) \homeq T_n^k(FG) \]
for all $k$ and $n$.

But now we use the fact that $F$ is finitary, i.e. preserves filtered homotopy colimits, to see that
\[ F(P_nG) = F(\hocolim_k T_n^kG) \homeq \hocolim_k F(T_n^kG) \homeq \hocolim_k T_n^k(FG) = P_n(FG) \]
as claimed.
\end{proof}

We conclude this section by showing that the full statement of Theorem \ref{thm:chainrule} follows from the special case in which $X = *$ and $F$ and $G$ are reduced, that is, $F(*) \homeq *$ and $G(*) \homeq *$.

\begin{remark} \label{rem:reduction}
Recall that
\[ \der_*F(GX) := \der_*(F_{GX}); \quad \der_*G(X) := \der_*(G_X); \quad \der_*(FG)(X) := \der_*((FG)_X) \]
where $F_{GX}$, $G_X$ and $(FG)_X$ are as in Definition \ref{def:derivative}. Note that $F_{GX}$ and $G_X$ are reduced functors. Also observe that
\[ \tag{*} \begin{split} F_{GX}(G_X(Z)) &= F_{GX}(\hofib(G(X \wdge Z) \to GX)) \\ &= \hofib \left[ \begin{diagram} \node{F(\hofib(G(X \wdge Z) \to GX) \wdge GX)} \arrow{s} \\ \node{FGX} \end{diagram} \right]. \end{split} \]
But the inclusion $X \to X \wdge Z$ determines an equivalence
\[ \hofib(G(X \wdge Z) \to GX) \wdge GX \weq G(X \wdge Z) \]
and so (*) is equivalent to
\[ \hofib(FG(X \wdge Z) \to FGX) = (FG)_X(Z). \]
Therefore we have $F_{GX} \circ G_X \homeq (FG)_X$.

The claim of Theorem \ref{thm:chainrule} can then be rewritten as
\[ \der_*(F_{GX} \circ G_X) \homeq \der_*(F_{GX}) \circ \der_*(G_X) \]
which is just the Theorem again applied to the derivatives at $*$ of the reduced functors $F_{GX}$ and $G_X$. For the remainder of the proof then, we can assume that $F$ and $G$ are reduced and consider only the Taylor towers at $*$.
\end{remark}

\section{The map that induces the chain rule equivalence} \label{sec:delta}

We prove Theorem \ref{thm:chainrule} in the reduced case by constructing, for each $n$, a natural transformation from $FG$ to a functor whose \ord{n} derivative is equivalent to the \ord{n} piece of the symmetric sequence $\der_*F \circ \der_*G$. In this section, we will construct this natural transformation and show that the \ord{n} derivative is as claimed. Then in \S\ref{sec:proof} we will prove that our map induces an equivalence on \ord{n} derivatives, and hence deduce Theorem \ref{thm:chainrule}.

To define these maps, we recall the definition of the \emph{co-cross-effect} of a homotopy functor. This is dual to Goodwillie's notion of cross-effect \cite[\S3]{goodwillie:2003} and was considered by McCarthy \cite[1.3]{mccarthy:2001} in studying dual calculus.

\begin{definition}[Co-cross-effects] \label{def:co-cross-effects}
Let $F: \spectra \to \spectra$ be a homotopy functor. The \emph{\ord{r} co-cross-effect} of $F$ is the functor of $r$ variables (i.e. from $\spectra^r$ to $\spectra$) defined by
\[ \creff^r(F)(X_1,\dots,X_r) := \hocofib\left( \hocolim_{J \varsubsetneq \{1,\dots,r\}} F(\prod_{j \in J} X_j) \to F(X_1 \times \dots \times X_r) \right) \]
In the language of \cite{goodwillie:1991} this is the total homotopy \emph{cofibre} of the $r$-cube given by applying $F$ to products of subsets of $\{X_1,\dots,X_r\}$, and is dual to the notion of cross-effect.
\end{definition}

\begin{lemma}
Let $F: \spectra \to \spectra$ be a homotopy functor. Then the \ord{r} co-cross-effect of $F$ is equivalent to the \ord{r} cross-effect, that is:
\[ \creff^r(F)(X_1,\dots,X_r) \homeq \creff_r(F)(X_1,\dots,X_r). \]
\end{lemma}
\begin{proof}
We see this by induction on $r$. The base case is $r = 1$ where we have fiber sequences
\[ F(*) \to F(X) \to \creff^1F(X), \text{ and } \creff_1F(X) \to F(X) \to F(*). \]
But each of these splits and so we have
\[ F(X) \homeq F(*) \wdge \creff^1(F) \homeq F(*) \times \creff_1(F) \]
from which it follows that the composite
\[ \creff_1F(X) \to F(X) \to \creff^1F(X) \]
is an equivalence.

Goodwillie shows in \cite{goodwillie:1991} that the total homotopy fibre that defines the \ord{r} cross-effect can be written as the homotopy fibre of a map between total homotopy fibres of cubes of one dimension less. Explicitly then, we get the following formula:
\[ \creff_r(F)(X_1,\dots,X_r) \homeq \hofib\left( \begin{diagram} \node{\creff_{r-1}(F(- \wdge X_r))(X_1,\dots,X_{r-1})}  \arrow{s} \\ \node{\creff_{r-1}(F)(X_1,\dots,X_{r-1})} \end{diagram} \right) \]
where the first term on the right-hand side is the \ord{r-1} cross-effect of the functor
\[ X \mapsto F(X \wdge X_r). \]
Similarly, for the co-cross-effect, we get
\[ \creff^r(F)(X_1,\dots,X_r) \homeq \hocofib\left( \begin{diagram} \node{\creff_{r-1}F(X_1,\dots,X_r)} \arrow{s} \\ \node{\creff^{r-1}(F(- \times X_r))(X_1,\dots,X_r)} \end{diagram} \right). \]
Since $X \wdge X_r \homeq X \times X_r$ for spectra, we get the desired result by a similar argument to the $r = 1$ case, and by the induction hypothesis, we get the desired result.
\end{proof}

\begin{definition} \label{def:cross-effect-maps}
The natural transformations we use to prove Theorem \ref{thm:chainrule} are based on the following construction:
\[ \Delta_r:  F(X) \arrow{e,t}{\Delta} F(X \times \dots \times X) \to \creff^r(F)(X,\dots,X). \]
The first map here is given by the diagonal map $X \to X \times \dots \times X$. The second map is the natural inclusion of $F(X \times \dots \times X)$ into the homotopy cofibre defining $\creff^r(F)(X,\dots,X)$.

Now let $\lambda$ be an element of $P(n)$, that is, an unordered partition of $n$ into positive integers, and recall that we can express $\lambda$ uniquely by writing $n$ as the sum
\[ n = k_1 l_1 + \dots + k_r l_r \]
where $k_i \geq 1$ and $l_1 < \dots < l_r$. We then define a map
\[ \begin{diagram} \node{\Delta_\lambda: FG} \arrow{e,t}{\Delta_r} \node{\creff^r(F)(G,\dots,G) \to [P_{k_1} \dots P_{k_r}] \creff^r(F)(P_{l_1}G, \dots, P_{l_r}G)} \end{diagram} \]
where the second map is given by applying $P_{k_i}$ in the \ord{i} position of the multivariable functor $\creff^r(F)$, and $P_{l_i}$ to the \ord{i} copy of $G$. For the sake of notation, we write $(F,G)_{\lambda}$ for the target of this map.
\end{definition}

\begin{remark}
The maps $\Delta_r$ form the basis of McCarthy's construction of the dual Taylor tower for functors from spectra to spectra \cite{mccarthy:2001}.
\end{remark}

\begin{definition} \label{def:delta}
Putting together the maps $\Delta_{\lambda}$ for all $\lambda \in P(n)$, we get
\[ \Delta: FG \to \prod_{\lambda \in P(n)} (F,G)_{\lambda}. \]
\end{definition}
In the next two sections, we will show that the map $\Delta$ determines an equivalence on \ord{n} derivatives. Theorem \ref{thm:chainrule} will then follow from the following calculation of the \ord{n} derivative of $(F,G)_{\lambda}$.

\begin{proposition} \label{prop:(FG)_l}
Let $F$ be a finitary homotopy functor and let $G$ be any reduced homotopy functor. Then there is an equivalence
\[ \der_n((F,G)_{\lambda}) \homeq (\Sigma_n)_+ \smsh_{H(\lambda)} \der_k(F) \smsh \der_{l_1}G^{\smsh k_1} \smsh \dots \smsh \der_{l_r}G^{\smsh k_r} \]
where the action of $H(\lambda)$ on the right-hand side is as in the definition of the composition product (Definition \ref{def:compprod}). This equivalence is equivariant with respect to the $\Sigma_n$-actions: the usual derivative action on the left, and the left regular action on the $(\Sigma_n)_+$ term on the right.
\end{proposition}

The proof of the proposition depends on the following two lemmas:
\begin{lemma} \label{lem:excisive}
Let $H: \spectra^r \to \spectra$ be a homotopy functor of $r$ variables that is $(k_1,\dots,k_r)$-excisive (that is, $k_i$-excisive in the \ord{i} variable). Let $G_1,\dots,G_r: \spectra \to \spectra$ be a sequence of homotopy functors such that $G_j$ is $l_j$-excisive. Then the composite functor
\[ H(G_1,\dots,G_r) : \spectra \to \spectra \]
is $(k_1 l_1 + \dots + k_r l_r)$-excisive.
\end{lemma}
\begin{proof}
Lemma 6.6 of \cite{goodwillie:2003} reduces this lemma to the case $r = 1$, so suppose $H: \spectra \to \spectra$ is $k$-excisive, and $G:\spectra \to \spectra$ is $l$-excisive. We need to show that $HG$ is $kl$-excisive.

By induction on $k$ using the fibre sequences
\[ (D_kH)G \to (P_kH)G \to (P_{k-1}H)G \]
we can reduce to the case where $H$ is homogeneous. We can then write $HG$ as
\[ \creff_kH(G,\dots,G)_{h\Sigma_k} \]
by \cite[3.5]{goodwillie:2003}. Setting $L = \creff_kH$, it is sufficient to show that if $L$ is a multilinear functor of $k$ variables, and $G$ is $l$-excisive, then $L(G,\dots,G)$ is $kl$-excisive. (The homotopy orbits commute with $P_{kl}$ so can be dropped.) Another application of \cite[6.6]{goodwillie:2003} then reduces to the case $k = 1$.

We therefore must show that if $L$ is a linear functor from $\spectra$ to $\spectra$, and $G$ is $l$-excisive, then $LG$ is also $l$-excisive. This, however, follows from the fact that a linear functor preserves homotopy cartesian $(l+1)$-cubes. (See the proof of Lemma \ref{lem:reduced}.)
\end{proof}

\begin{lemma} \label{lem:crosseffect}
Let $F:\spectra \to \spectra$ be a homotopy functor. Then
\[ [D_{k_1} \dots D_{k_r}] \creff^r(F)(X_1,\dots,X_r) \homeq D_1^{(k)} \creff^k(F)(\overbrace{X_1,\dots,X_1}^{k_1},\dots,)_{h\Sigma_{k_1} \times \dots \times \Sigma_{k_r}} \]
where $k = k_1+\dots+k_r$.
\end{lemma}
\begin{proof}
This is an slight generalization of Goodwillie's formula that for a general functor $F$ we have an equivalence
\[ D_nF(X) \homeq D_1^{(n)}\creff^n(F)(X,\dots,X)_{h\Sigma_n} \]
where $D_1^{(n)}$ denotes the multilinearization of the $n$-variable functor $\creff^n(F)$. We apply this formula for each of the variables of $\creff^r(F)$ and find that the left-hand side of the statement in the lemma is equivalent to
\[ D_1^{(k)}([\creff^{k_1} \dots \creff^{k_r}] \creff^r(F))(\overbrace{X_1,\dots,X_1}^{k_1},\dots,\overbrace{X_r,\dots,X_r}^{k_r})_{h\Sigma_{k_1} \times \dots \times \Sigma_{k_r}} \]
where we have taken the \ord{k_i} co-cross-effect in the \ord{i} variable of $\creff^r(F)$ separately and linearized with respect to each of the resulting $k$ variables. The iterated co-cross-effect is an iterated total homotopy cofibre which is equivalent to a single large total homotopy cofibre, namely the \ord{k} co-cross-effect of $F$. This gives the result.
\end{proof}

\begin{proof}[Proof of Proposition \ref{prop:(FG)_l}]
Recall that $(F,G)_{\lambda}$ is the target of the map $\Delta_{\lambda}$ of Definition \ref{def:cross-effect-maps}, that is,
\[ (F,G)_{\lambda} := [P_{k_1} \dots P_{k_r}]\creff^r(F)(P_{l_1}G,\dots,P_{l_r}G). \]
Notice that by Lemma \ref{lem:excisive}, $(F,G)_{\lambda}$ is $n = (k_1l_1+\dots+k_rl_r)$-excisive.

Now the map
\[ [D_{k_1} P_{k_2} \dots P_{k_r}] \creff^r(F)(P_{l_1}G,\dots,P_{l_r}G) \to [P_{k_1} P_{k_2} \dots P_{k_r}] \creff^r(F)(P_{l_1}G,\dots,P_{l_r}G) \]
is an equivalence on \ord{n} derivatives since it fits in a fibre sequence with
\[ [P_{k_1-1} P_{k_2} \dots P_{k_r}] \creff^r(F)(P_{l_1}G,\dots,P_{l_r}G) \]
which is $(n - l_1)$-excisive (again by Lemma \ref{lem:excisive}). Repeating for $k_2$ and so on, we see that the map
\[ \tag{*} [D_{k_1} \dots D_{k_r}] \creff^r(F)(P_{l_1}G,\dots,P_{l_r}G) \to (F,G)_{\lambda} \]
is an equivalence on \ord{n} derivatives.

By Lemma \ref{lem:crosseffect}, the left-hand side of (*) is equivalent to
\[ D_1^{(k)} \creff^k(F)(P_{l_1}G,\dots,P_{l_1}G,\dots,P_{l_r}G,\dots,P_{l_r}G)_{h\Sigma_{k_1} \times \dots \times \Sigma_{k_r}}. \]
Now when we replace $P_{l_1}G$ with the terms from
\[ D_{l_1}G \to P_{l_1}G \to P_{l_1-1}G \]
we get a fibre sequence (since linear functors of spectra preserve fibre sequences). The term in this sequence involving $P_{l_1-1}G$ is $n-1$-excisive by Lemma \ref{lem:excisive}, and so the map $D_{l_1}G \to P_{l_1}G$ induces an equivalence on \ord{n} derivatives. By the same argument, we can replace each $P_{l_j}G$ with $D_{l_j}G$.

We have therefore concluded that the \ord{n} derivative of $(F,G)_{\lambda}$ is equivalent to that of
\[ D_1^{(k)} \creff^k(F)(D_{l_1}G,\dots,D_{l_r}G)_{h\Sigma_{k_1} \times \dots \times \Sigma_{k_r}}. \]
We can now use the usual formulas for $D_1^{(k)} \creff^k(F)$ and $D_{l_i}G$ in terms of the derivatives. We see that the above functor is equivalent to
\[ (\der_kF \smsh (\der_{l_1}G \smsh_{h\Sigma_{l_1}} X^{\smsh l_1})^{\smsh k_1} \smsh \dots \smsh (\der_{l_r}G \smsh_{h\Sigma_{l_r}} X^{\smsh l_r})^{\smsh k_r})_{h\Sigma_{k_1} \times \dots \times \Sigma_{k_r}} \]
for finite $X$. This uses the fact that $F$ is finitary since we are applying the usual derivative formula for $D_1^{(k)}\creff^k(F)$ to inputs that are not necessarily finite spectra.

Taking smash products preserves homotopy orbits, so we can write this as
\[ \left[(\der_kF \smsh \der_{l_1}G^{\smsh k_1} \smsh \dots \smsh \der_{l_r}G^{\smsh k_r} \smsh X^{\smsh n})_{h\Sigma_{l_1}^{k_1} \times \dots \times \Sigma_{l_r}^{kr}} \right]_{h\Sigma_{k_1} \times \dots \times \Sigma_{k_r}}. \]
We can collect together these homotopy orbit constructions and write this instead as
\[ (\der_kF \smsh \der_{l_1}G^{\smsh k_1} \smsh \dots \smsh \der_{l_r}G^{\smsh k_r} \smsh X^{\smsh n})_{hH(\lambda)} \]
where we recall that $H(\lambda)$ is the subgroup $(\Sigma_{l_1} \wr \Sigma_{k_1} \times \dots \times \Sigma_{l_r} \wr \Sigma_{k_r})$ of $\Sigma_n$. The \ord{n} derivative of this is then
\[ (\Sigma_n)_+ \smsh_{H(\lambda)} \der_k(F) \smsh \der_{l_1}G^{\smsh k_1} \smsh \dots \smsh \der_{l_r}G^{\smsh k_r} \]
which completes the proof.
\end{proof}

\section{The Taylor tower of FG when F is homogeneous} \label{sec:homo}

The remainder of the proof of Theorem \ref{thm:chainrule} involves showing that the map $\Delta$ of Definition \ref{def:delta} is an equivalence on \ord{n} derivatives. Our method is to prove it first in the case that $F$ is $k$-homogeneous for some $k$ and then use induction on the Taylor tower to prove it for all $F$. Those proofs will occupy \S\ref{sec:proof}. In this section, we set the scene for the proof when $F$ is homogeneous by describing the full Taylor tower of $FG$ in that case. 

Goodwillie shows in \cite[\S3]{goodwillie:2003} that any $k$-homogeneous functor $F: \spectra \to \spectra$ is of the form
\[ F(X) \homeq L(X,\dots,X)_{h\Sigma_k} \]
where $L$ is a symmetric multilinear functor of $k$ variables. Since $P_n$ commutes with the homotopy orbit construction (for spectrum-valued functors), these will be easy to deal with separately. We therefore start by considering the Taylor tower of $FG$ where $F$ is of the form
\[ F(X) \homeq L(X,\dots,X) \]
and $L: \spectra^k \to \spectra$ is multilinear.

We should stress that as with most of this results in this paper, it is essential to assume that $F$, and hence $L$, is finitary. Example \ref{ex:counterexample} provides a counterexample to the results of this section when that condition is dropped.

We start by defining functors that will turn out to form the Taylor tower of $FG$.

\begin{definition} \label{def:p_n(FG)}
Let $F$ be a homotopy functor of the form
\[ F(X) := L(X,\dots,X) \]
where $L: \spectra^k \to \spectra$ is finitary and multilinear, and let $G: \spectra \to \spectra$ be a reduced homotopy functor.

Let $\Pi_k(n)$ denote the category whose objects are \emph{ordered} $k$-tuples $(r_1,\dots,r_k)$ of positive integers with the property that $r_1 + \dots + r_k \leq n$ and such that there is a unique morphism from $(r_1,\dots,r_k)$ to $(s_1,\dots,s_k)$ if and only if $r_i \geq s_i$ for all $i$.

We define a diagram $\mathcal{X}_n$ of spectra indexed by $\Pi_k(n)$ as follows:
\[ \mathcal{X}_n(r_1,\dots,r_k) := L(P_{r_1}G,\dots,P_{r_k}G) \]
and the morphisms in the diagram come from the structure maps in the Taylor tower for $G$.

We then define functors $p_n(FG): \spectra \to \spectra$ by taking the homotopy limit of the diagrams $\mathcal{X}_n$:
\[ p_n(FG)(X) := \holim_{\Pi_k(n)}\mathcal{X}_n = \holim_{(r_1,\dots,r_k)} L(P_{r_1}G(X),\dots,P_{r_k}G(X)). \]
If $\iota: \Pi_k(n-1) \to \Pi_k(n)$ denotes the obvious inclusion, then there is a natural map
\[ \iota^*\mathcal{X}_n \to \mathcal{X}_{n-1} \]
which determines a restriction map
\[ p_n(FG) \to p_{n-1}(FG). \]
The natural maps $FG = L(G,\dots,G) \to L(P_{r_1}G,\dots,P_{r_k}G)$ then assemble to form maps
\[ FG \to p_n(FG) \]
that commute with the restrictions.
\end{definition}

Our main aim now is to show that the sequence $\{p_n(FG)\}$ is equivalent to the Taylor tower of the functor $FG = L(G,\dots,G)$. The key to this is a calculation of the homotopy fibres of the maps $p_n(FG) \to p_{n-1}(FG)$.

\begin{definition} \label{def:d_n(FG)}
Let $d_n(FG)$ denote the homotopy fibre of the restriction map:
\[ d_n(FG) := \hofib(p_n(FG) \to p_{n-1}(FG)). \]
\end{definition}

\begin{prop} \label{prop:d_n(FG)}
There is a natural equivalence
\[ \alpha: \prod_{r_1+\dots+r_k = n} L(D_{r_1}G,\dots,D_{r_k}G) \weq d_n(FG). \]
\end{prop}
\begin{proof}
To construct this equivalence, it is convenient to assume that the maps $P_rG \to P_{r-1}G$ in the Taylor tower for $G$ are fibrations, and the layers $D_rG$ are taken to be the \emph{strict} fibres of these maps. This means that the composite
\[ D_rG \to P_rG \to P_{r-1}G \]
is the trivial map.

Now define a diagram $\mathcal{D}_n$ of functors indexed on $\Pi_k(n)$ by
\[ \mathcal{D}_n(r_1,\dots,r_k) := \begin{cases} L(D_{r_1}G,\dots,D_{r_k}G) & \text{if $r_1+\dots+r_k = n$}, \\ * & \text{otherwise}. \end{cases} \]
There is a map of diagrams
\[ \mathcal{D}_n \to \mathcal{X}_n .\]
induced by the natural transformations $D_{r_i}G \to P_{r_i}G$. The induced map of homotopy limits is then a map
\[ \tilde{\alpha}: \prod_{r_1+\dots+r_k = n} L(D_{r_1}G,\dots,D_{r_k}G) \to p_n(FG). \]
Now the composite of $\tilde{\alpha}$ with the restriction $p_n(FG) \to p_{n-1}(FG)$ is the trivial map since it is built from the composites $D_{r_i}G \to P_{r_i}G \to P_{r_i-1}G$. Therefore $\tilde{\alpha}$ factors via the homotopy fibre of that restriction. This gives us the required map
\[ \alpha: \prod_{r_1+\dots+r_k = n} L(D_{r_1}G,\dots,D_{r_k}G) \to d_n(FG). \]

To show that $\alpha$ is an equivalence, we will construct an inverse for it in the homotopy category. Suppose that $r_1+\dots+r_k = n$ and consider the following diagram:
\[ \begin{diagram}
  \node{d_n(FG)} \arrow{e} \arrow{se,b}{\delta} \node{p_n(FG)} \arrow{s} \arrow{e} \node{p_{n-1}(FG)} \arrow{s} \\
  \node[2]{L(P_{r_1}G,\dots,P_{r_k}G)} \arrow{e,t}{\gamma} \node{L(P_{r_1}G,\dots,P_{r_k-1}G)}
\end{diagram} \]
where the vertical maps are projections from the homotopy limits. The square commutes up to homotopy which implies that $\gamma \delta$ is nullhomotopic. Therefore, in the homotopy category, $\delta$ factors via the homotopy fibre of $\gamma$. Since $L$ is linear in each variable, that fibre is equivalent to $L(P_{r_1}G,\dots,P_{r_{k-1}}G,D_{r_k}G)$ and so we have constructed a map
\[ d_n(FG) \to L(P_{r_1}G,\dots,P_{r_{k-1}}G,D_{r_k}G). \]

By a similar argument, the composite
\[ d_n(FG) \to L(P_{r_1}G,\dots,P_{r_{k-1}}G,D_{r_k}G) \to L(P_{r_1}G,\dots,P_{r_{k-1}-1}G,D_{r_k}G) \]
is nullhomotopic, and so factors via $L(P_{r_1}G,\dots,D_{r_{k-1}}G,D_{r_k}G)$, and so on. By induction, we can factor via all the $D_{r_i}G$ and putting the resulting maps together for all $k$-tuples with $r_1+\dots+r_k = n$, we get the required map (in the homotopy category)
\[ \beta: d_n(FG) \to \prod_{r_1+\dots+r_k = n} L(D_{r_1}G,\dots,D_{r_k}G). \]

We now claim that $\beta$ is inverse to $\alpha$ in the homotopy category. The composite $\beta\alpha$ is easily seen from the definitions to be the identity. The composite
\[ \begin{diagram}
  \node{d_n(FG)} \arrow{e,t}{\beta} \node{\prod_{r_1+\dots+r_k = n} L(D_{r_1}G,\dots,D_{r_k}G)} \arrow{e,t}{\alpha} \node{d_n(FG)}
\end{diagram} \]
is the identity because we have constructed $\beta$ in such a way that the following diagram commutes in the homotopy category:
\[ \begin{diagram}
  \node{d_n(FG)} \arrow{e} \arrow{s} \node{p_n(FG)} \arrow{s} \\
  \node{L(D_{r_1}G,\dots,D_{r_k}G)} \arrow{e} \node{L(P_{r_1}G,\dots,P_{r_k}G)}
\end{diagram} \]
This completes the proof.
\end{proof}

\begin{corollary} \label{cor:p_n(FG)}
For each $n$, $d_n(FG)$ is $n$-homogeneous and $p_n(FG)$ is $n$-excisive.
\end{corollary}
\begin{proof}
We know that $L(D_{r_1}G,\dots,D_{r_k}G)$ is $n$-excisive by Lemma \ref{lem:excisive}. By Lemma \ref{lem:reduced}, the functor
\[ (X_1,\dots,X_k) \mapsto L(D_{r_1}G(X_1), \dots, D_{r_k}G(X_k)) \]
is $r_i$-reduced in the \ord{i} variable. Therefore, by Lemma 6.10 of \cite{goodwillie:2003}, the functor
\[ X \mapsto L(D_{r_1}G(X),\dots,D_{r_k}G(X)) \]
is $r_1+\dots+r_k = n$-excisive. Then $d_n(FG)$ is a product of $n$-homogeneous functors, and so it $n$-homogeneous.

Now note that $p_1(FG)$ is only nontrivial if $k = 1$ in which case it is equivalent to $L(D_1G)$ which is $1$-excisive by Lemma \ref{lem:excisive}. So by induction, we may assume that $p_{n-1}(FG)$ is $(n-1)$-excisive, and hence $n$-excisive. The fibration sequence
\[ d_n(FG) \to p_n(FG) \to p_{n-1}(FG) \]
then implies that $p_n(FG)$ is $n$-excisive too.
\end{proof}

The next lemma says that $p_r(FG)$ is the $r$-excisive part of $p_n(FG)$ when $r \leq n$.

\begin{lemma} \label{lem:p_n(FG)}
For $r \leq n$, we have
\[ P_r(p_n(FG)) \homeq p_r(FG). \]
\end{lemma}
\begin{proof}
We do induction on $n$. The case $n = r$ says that $p_r(FG)$ is $r$-excisive which is Corollary \ref{cor:p_n(FG)}. The induction step then follows from Proposition \ref{prop:p_n(FG)} since we have
\[ P_r(d_n(FG)) \homeq * \]
for $n > r$.
\end{proof}

We are now able to deduce that the functors $p_n(FG)$ form the Taylor tower of $FG$.

\begin{proposition} \label{prop:p_n(FG)}
Let $F$ be a finitary homotopy functor of the form
\[ F(X) = L(X,\dots,X) \]
where $L:\spectra^k \to \spectra$ is multilinear, and let $G: \spectra \to \spectra$ be any reduced homotopy functor. Then the sequence
\[ FG \to \dots \to p_n(FG) \to p_{n-1}(FG) \to \dots \]
is equivalent to the Taylor tower of $FG$. Therefore also, $D_n(FG) \homeq d_n(FG)$.
\end{proposition}
\begin{proof}
The maps $P_n(FG) \to P_n(p_n(FG)) \leftarrow p_n(FG)$ commute with the tower restriction maps and so it is sufficient to show that these are equivalences. Corollary \ref{cor:p_n(FG)} tells us that the right-hand maps here are equivalences, so it remains to show that $P_n(FG) \weq P_n(p_n(FG))$.

Suppose first that $G$ is $m$-excisive for some $m$. Consider the diagram $\mathcal{X}_n$ of Definition \ref{def:p_n(FG)} whose homotopy limit is $p_n(FG)$, and suppose that $n \geq mk$. If $r_i > m$, then the map
\[ \mathcal{X}_n(r_1,\dots,r_k) \to \mathcal{X}_n(r_1,\dots,r_i-1,\dots,r_k) \]
is an equivalence. It follows that the homotopy limit of $\mathcal{X}_n$ is equivalent to the homotopy limit of its restriction to the subcategory of $\Pi_k(n)$ whose objects are $k$-tuples $(r_1,\dots,r_k)$ with $r_1+\dots+r_k \leq n$ and $r_i \leq m$ for all $i$. This restricted diagram has an initial object $L(P_mG,\dots,P_mG)$ and so it follows that
\[ p_n(FG) \homeq L(P_mG,\dots,P_mG) \homeq L(G,\dots,G) = FG, \]
that is, the map
\[ FG \weq p_n(FG) \]
is an equivalence for $n \geq mk$. It follows that
\[ P_n(FG) \weq P_n(p_n(FG)) \]
for $n \geq mk$. For $n < mk$, we have
\[ P_n(FG) \weq P_n(p_{mk}(FG)) \weq P_n(p_n(FG)) \]
where the second equivalence comes from Lemma \ref{lem:p_n(FG)}.

Now suppose that $G$ is arbitrary. Then we have a commutative diagram
\[\begin{diagram}
  \node{P_n(FG)} \arrow{s} \arrow{e,t}{\sim} \node{P_n(FP_nG)} \arrow{s,r}{\sim} \\
  \node{P_n(p_n(FG))} \arrow{e,t}{\sim} \node{P_n(p_n(FP_nG))}
\end{diagram} \]
We have just shown that the right-hand vertical map is an equivalence. The bottom map is an equivalence since the diagrams $\mathcal{X}_n$ that define $p_n(FG)$ and $p_n(FP_nG)$ are termwise equivalence. Finally, the top map is an equivalence by part (2) of Proposition \ref{prop:fundamental}. Therefore, the left-hand vertical map is an equivalence and we are done.
\end{proof}

We can now take homotopy orbits to get the following explicit formulas for the Taylor tower of $FG$ for any finitary homogeneous functor $F$ and reduced $G$.

\begin{theorem} \label{thm:hom}
Let $F: \spectra \to \spectra$ be a $k$-homogeneous finitary homotopy functor, and let $G: \spectra \to \spectra$ be any reduced homotopy functor. Then we have
\[ P_n(FG) \homeq \left[\holim_{r_1+\dots+r_k \leq n} \der_kF \smsh P_{r_1}G \smsh \dots \smsh P_{r_k}G \right]_{h\Sigma_k} \]
and
\[ D_n(FG) \homeq \left[ \prod_{r_1+\dots+r_k = n} \der_kF \smsh D_{r_1}G \smsh \dots \smsh D_{r_k}G \right]_{h\Sigma_k} \]
where $\Sigma_k$ acts in each case by permuting $(r_1,\dots,r_k)$. The homotopy limit in the formula for $P_n(FG)$ is taken over the category $\Pi_k(n)$.
\end{theorem}
\begin{proof}
Applying Propositions \ref{prop:d_n(FG)} and \ref{prop:p_n(FG)} to $L = \creff_k(F)$, and noting that
\[ \creff_k(F)(X_1,\dots,X_r) \homeq \der_kF \smsh X_1 \smsh \dots \smsh X_r \]
we obtain formulas for $P_n(\creff_k(F)(G,\dots,G))$ and $D_n(\creff_k(F)(G,\dots,G))$. We then apply the homotopy orbits construction which commutes with $P_n$ and $D_n$ and gives the claimed formulas since
\[ FG \homeq \creff_k(F)(G,\dots,G)_{h\Sigma_k} \]
when $F$ is $k$-homogeneous.
\end{proof}

\begin{remark}
The only part of this section that does not hold in an arbitrary stable model category is the way we write the formulas in Theorem \ref{thm:hom} using smash products. In general, we can just write them in terms of the \ord{k} cross-effect of $F$, that is:
\[ P_n(FG) \homeq \left[\holim_{r_1+\dots+r_k \leq n} \creff_k(F)(P_{r_1}G,\dots,P_{r_k}G) \right]_{h\Sigma_k} \]
and
\[ D_n(FG) \homeq \left[ \prod_{r_1+\dots+r_k = n} \creff_k(F)(D_{r_1}G,\dots,D_{r_k}G) \right]_{h\Sigma_k}. \]
\end{remark}

\section{Proof of the chain rule} \label{sec:proof}

The aim of this section is to complete the proof of the chain rule by showing that the map $\Delta$ of Definition \ref{def:delta} induces an equivalence on \ord{n} derivatives. We first use the results of \S\ref{sec:homo} to do this in the case that $F$ is homogeneous, and then do it for general $F$ by induction on the Taylor tower.

\begin{proposition} \label{prop:equiv_hom}
Let $F:\spectra \to \spectra$ be a $k$-homogeneous finitary homotopy functor and let $G:\spectra \to \spectra$ be any reduced homotopy functor. Then the map
\[ D_n\Delta: D_n(FG) \weq \prod_{\lambda \in P(n)} D_n((F,G)_{\lambda}) \]
is an equivalence, where $\Delta$ is as in Definition \ref{def:delta}.
\end{proposition}
\begin{proof}
We start by taking $F$ to be of the form $F(X) = L(X,\dots,X)$ for a multilinear functor $L$ of $k$ variables.

Since $L$ is linear in each variable (and since finite products and coproducts are equivalent in $\spectra$), we have
\[ L(X_1 \times \dots \times X_r,\dots,X_1 \times \dots \times X_r) \homeq \prod_{s:\{1,\dots,k\} \to \{1,\dots,r\}} L(X_{s(1)},\dots,X_{s(k)}). \]
where the product is indexed by the set of functions $s$ from $\{1,\dots,k\}$ to $\{1,\dots,r\}$. The co-cross-effect is then given by
\[ \creff^r(F)(X_1,\dots,X_r) \homeq \prod_{s:\{1,\dots,k\} \fib \{1,\dots,r\}} L(X_{s(1)},\dots,X_{s(k)}) \]
where now the product is over all \emph{surjective} functions. Applying $[P_{k_1},\dots,P_{k_r}]$ to this kills all the terms that have more than $k_j$ copies of $X_j$, so we get
\[ [P_{k_1} \dots P_{k_r}] \creff^r(F)(X_1,\dots,X_r) \homeq  \prod_{\substack{s: \{1,\dots,k\} \to \{1,\dots,r\} \\ 1 \leq |s^{-1}(j)| \leq k_j}} L(X_{s(1)},\dots,X_{s(k)}). \]
Now let $\lambda \in P(n)$ be a partition of the positive integer $n$ represented by the decomposition $n = k_1 l_1 + \dots + k_r l_r$. Then we have
\[ (F,G)_{\lambda} = \prod_{\substack{s: \{1,\dots,k\} \to \{1,\dots,r\} \\ 1 \leq |s^{-1}(j)| \leq k_j}} L(P_{l_{s(1)}}G,\dots,P_{l_{s(k)}}G). \]
With respect to this expression the map $\Delta_{\lambda}: FG \to (F,G)_{\lambda}$ of Definition \ref{def:cross-effect-maps} is given by assembling the maps
\[ \tag{*} FG = L(G,\dots,G) \to L(P_{l_{s(1)}}G,\dots,P_{l_{s(k)}}G) \]
where $G \to P_{l_i}G$ comes from the Taylor tower of $G$.

We now want to apply $D_n$ to $\Delta_{\lambda}$. To help simplify this, notice the following facts:
\begin{itemize}
  \item $l_{s(1)}+\dots+l_{s(k)} \leq k_1 l_1 + \dots + k_r l_r = n$ because $|s^{-1}(j)| \leq k_j$, with equality if and only if $|s^{-1}(j)| = k_j$ for $j = 1,\dots,r$;
  \item if $l_{s(1)}+\dots+l_{s(k)} < n$ then $D_n(L(P_{l_{s(1)}}G,\dots,P_{l_{s(k)}}G)) \homeq *$, by Lemma \ref{lem:excisive};
  \item if $l_{s(1)}+\dots+l_{s(k)} = n$ then
  \[ D_n(L(P_{l_{s(1)}}G,\dots,P_{l_{s(k)}}G)) \homeq L(D_{l_{s(1)}}G,\dots,D_{l_{s(k)}}G), \]
  again by Lemma \ref{lem:excisive} using the fibre sequences
  \[ D_{l_{s(i)}}G \to P_{l_{s(i)}}G \to P_{l_{s(i)}-1}G. \]
\end{itemize}
Therefore:
\[ D_n((F,G)_{\lambda}) \homeq \prod_{\substack{s:\{1,\dots,k\} \to \{1,\dots,r\} \\ |s^{-1}(j)| = k_j}} L(D_{l_{s(1)}}G,\dots,D_{l_{s(k)}}G) \]
Now given a function $s:\{1,\dots,k\} \to \{1,\dots,r\}$ as in the indexing set of the product, define
\[ r_i := l_{s(i)}. \]
Then $(r_1,\dots,r_k)$ is an ordered $k$-tuples such that $r_1+\dots+r_k = n$. This $k$-tuple is of \emph{type} $\lambda$ in the sense that precisely $k_j$ of the terms are equal to $l_j$. Conversely, given such a $k$-tuple, define a function $s: \{1,\dots,k\} \to \{1,\dots,r\}$ by
\[ s(i) := j \text{ where $l_j = r_i$}. \]
Such a $j$ exists because $(r_1,\dots,r_k)$ is of type $\lambda$ and is unique because all the $l_j$ are distinct. These constructions set up a 1-1 correspondence between the indexing set of the above product, and the set of ordered $k$-tuples $(r_1,\dots,r_k)$ of `type' $\lambda$ (i.e. $k_i$ of the terms are equal to $l_i$ for $i = 1,\dots,r$). Therefore, we can write
\[ D_n((F,G)_{\lambda}) \homeq \prod_{(r_1,\dots,r_k) \text{ of type $\lambda$}} L(D_{r_1}G,\dots,D_{r_k}G). \]
But now by Proposition \ref{prop:d_n(FG)} we have
\[ D_n(FG) \homeq \prod_{r_1+\dots+r_k = n} L(D_{r_1}G,\dots,D_{r_k}G) \homeq \prod_{\lambda \in P(n)} D_n((F,G)_{\lambda}) \]
and since the maps $\Delta_{\lambda}$ are defined as in (*) in terms of the projections of the Taylor tower of $G$, the map $\Delta$ expresses the above equivalence, as required.

Finally, to deduce the result for a general $F$, we take $L$ to be the \ord{k} cross-effect of $F$ and apply homotopy orbits with respect to the action of $\Sigma_k$ determined by the symmetry isomorphisms for the cross-effect. Taking these homotopy orbits commutes with all the relevant constructions, namely $D_n$, $P_{k_i}$ and the co-cross-effects, so we are done.
\end{proof}

\begin{prop} \label{prop:equiv}
The conclusion of Proposition \ref{prop:equiv_hom} holds for any finitary homotopy functor $F$.
\end{prop}
\begin{proof}
The map $F \to P_nF$ determines a commutative diagram
\[ \begin{diagram}
  \node{D_n(FG)} \arrow{e,t}{D_n\Delta} \arrow{s,l}{\sim} \node{\prod_{\lambda \in P(n)} D_n((F,G)_{\lambda})} \arrow{s,l}{\sim} \\
  \node{D_n((P_nF)G)} \arrow{e,t}{D_n\Delta} \node{\prod_{\lambda \in P(n)} D_n((P_nF,G)_{\lambda})}
\end{diagram} \]
The left-hand vertical map is an equivalence by part (1) of Proposition \ref{prop:fundamental}. The right-hand vertical map is an equivalence by Proposition \ref{prop:(FG)_l} since the map $F \to P_nF$ is an equivalence on \ord{k} derivatives if $k \leq n$. We are therefore reduced to showing that the bottom map is an equivalence.

We do this by induction on the Taylor tower of $F$. The base case is when $F$ is constant in which case both source and target of $\Delta$ are trivial. Consider then the diagram
\[ \begin{diagram}
  \node{D_n((D_kF)G)} \arrow{e,t}{D_n\Delta} \arrow{s} \node{\prod_{\lambda \in P(n)} D_n((D_kF,G)_{\lambda})} \arrow{s} \\
  \node{D_n((P_kF)G)} \arrow{e,t}{D_n\Delta} \arrow{s} \node{\prod_{\lambda \in P(n)} D_n((P_kF,G)_{\lambda})} \arrow{s} \\
  \node{D_n((P_{k-1}F)G)} \arrow{e,t}{D_n\Delta} \node{\prod_{\lambda \in P(n)} D_n((P_{k-1}F,G)_{\lambda})}
\end{diagram} \]
But $D_n\Delta$ is an equivalence for $D_kF$ by Proposition \ref{prop:equiv_hom}, and for $P_{k-1}F$ by the induction hypothesis. Hence it is an equivalence for $P_kF$. This completes the proof.
\end{proof}

Finally, we deduce our chain rule.

\begin{thm:chainrule}
Let $F,G:\spectra \to \spectra$ be homotopy functors and suppose that $F$ preserves filtered homotopy colimits. Then
\[ \der_*(FG)(X) \homeq \der_*(F)(GX) \circ \der_*(G)(X). \]
\end{thm:chainrule}
\begin{proof}
The case where $F$ and $G$ are reduced and $X = *$ follows from Propositions \ref{prop:equiv} and \ref{prop:(FG)_l} using the definition of the composition product (\ref{def:compprod}) and the fact that finite products and coproducts of spectra are equivalent. The general case then follows from the argument of Remark \ref{rem:reduction}.
\end{proof}

\section{Chain rule for Taylor towers of functors of spectra} \label{sec:tower}

We now use the results on derivatives to say something about the full Taylor tower for a composite functor $FG$. Unfortunately, we are unable to produce a simple formula for $P_n(FG)$ in terms of the Taylor towers of $F$ and $G$. Instead our main result gives a recursive way to obtain expressions for the $P_n(FG)$ as homotopy limits of various diagrams. We carry out this process for $n = 1,2,3$. The usefulness (or otherwise) of our approach depends on having good models for the objects and maps in these diagrams.

We again consider the expansion of the Taylor tower of $FG$ about a general base object $X$, but we need to give another word of caution. We are only able to describe $P^X_n(FG)$ as a functor on $\spectra_X$, that is, on maps $Y \to X$ that have a section. Goodwillie's general theory told us that for the layer $D^X_n(FG)$, this was sufficient to determine its value on $\spectra/X$, that is, on all maps $Y \to X$. For $P_n$ this is not the case.

As with the chain rule for derivatives, the first order of business is to reduce to the case where $F$ and $G$ are reduced functors, and $X = *$. The following lemma helps us to do that.

\begin{lemma} \label{lem:P^X_n(FG)}
If $Y \to X$ has a section, then
\[ P^X_nF(Y) \homeq P_n(F_X)(\hofib(Y \to X)) \wdge F(X) \]
where $F_X(Z) = \hofib(F(Z \wdge X) \to FX)$ as in Definition \ref{def:derivative}.
\end{lemma}
\begin{proof}
When $Y \to X$ has a section, so does $P^X_nF(Y) \to P^X_nF(X) = F(X)$ and hence
\[ \tag{*} P^X_nF(Y) \homeq \hofib(P^X_nF(Y) \to FX) \wdge FX. \]
But by Lemma \ref{lem:hofib} we have
\[ \begin{split} P_n(F_X(\hofib(Y \to X))) & \homeq P^X_n(Y \mapsto F_X(\hofib(Y \to X))) \\
                                           & \homeq P^X_n(Y \mapsto \hofib(FY \to FX)) \\
                                           & \homeq \hofib(P^X_nF(Y) \to FX) \end{split} \]
which combined with (*) yields the lemma.
\end{proof}

Along similar lines to Remark \ref{rem:reduction}, we now notice that
\[ \begin{split} P^X_n(FG)(Y) &\homeq P_n((FG)_X)(\hofib(Y \to X)) \wdge FG(X) \\
                              &\homeq P_n(F_{GX}G_X)(\hofib(Y \to X)) \wdge FG(X) \end{split} \]
It is therefore sufficient to study the case where $F$ and $G$ are reduced and $X = *$. Applying this to the functors $F_{GX}$ and $G_X$ will allow us to say something about $P^X_n(FG)$.

The question before us then is to get information about $P_n(FG)$ from the individual Taylor towers of $F$ and $G$. Our main result is the following. Notice that as elsewhere in this paper, we need the condition that $F$ is finitary.

\begin{proposition}
Let $F,G: \spectra \to \spectra$ be reduced homotopy functors with $F$ finitary, and let $(F,G)_{\lambda}$ and the maps $\Delta_{\lambda}: FG \to (F,G)_{\lambda}$ be as in Definition \ref{def:cross-effect-maps}. Then the following diagram is a homotopy pullback:
\[ \begin{diagram}
  \node{P_n(FG)} \arrow{e} \arrow{s,l}{\prod \Delta_{\lambda}} \node{P_{n-1}(FG)} \arrow{s,l}{\prod \Delta_{\lambda}} \\
  \node{\prod_{\lambda \in P(n)} P_n((F,G)_{\lambda})} \arrow{e} \node{\prod_{\lambda \in P(n)} P_{n-1}((F,G)_{\lambda})}
\end{diagram} \]
\end{proposition}
\begin{proof}
The above diagram is part of the following map of fibre sequences:
\[ \begin{diagram}
  \node{D_n(FG)} \arrow{e} \arrow{s,l}{\sim} \node{P_n(FG)} \arrow{e} \arrow{s} \node{P_{n-1}(FG)} \arrow{s} \\
  \node{\prod_{\lambda \in P(n)} D_n((F,G)_{\lambda})} \arrow{e} \node{\prod_{\lambda \in P(n)} P_n((F,G)_{\lambda})} \arrow{e} \node{\prod_{\lambda \in P(n)} P_{n-1}((F,G)_{\lambda})}
\end{diagram} \]
The left-hand vertical map is an equivalence by Proposition \ref{prop:equiv} and hence the right-hand square is a homotopy pullback as claimed.
\end{proof}

We now give calculations for small $n$.

\begin{examples}
For $n=1$, we get the following simple expression
\[ P_1(FG) \homeq (P_1F)(P_1G) \]
which is clearly analogous to the chain rule for first derivatives in ordinary calculus.

The case $n = 2$ is also manageable. There are two partitions of $n = 2$, corresponding to $2 = 2$ and $2 = 1 + 1$.
We have $(FG)_2 \homeq (P_1F)(P_2G)$ and $(FG)_{1+1} \homeq (P_2F)(P_1G)$. Each of these has $P_1((F,G)_{\lambda}) \homeq (P_1F)(P_1G)$. Therefore, we get a homotopy pullback
\[ \begin{diagram}
  \node{P_2(FG)} \arrow{e} \arrow{s} \node{(P_1F)(P_1G)} \arrow{s,r}{\Delta} \\
  \node{(P_2F)(P_1G) \times (P_1F)(P_2G)} \arrow{e} \node{(P_1F)(P_1G) \times (P_1F)(P_1G)}
\end{diagram}\]
This is equivalent to the existence of a homotopy pullback of the form
\[ \begin{diagram}
  \node{P_2(FG)} \arrow{e} \arrow{s} \node{(P_2F)(P_1G)} \arrow{s} \\
  \node{(P_1F)(P_2G)} \arrow{e} \node{(P_1F)(P_1G)}
\end{diagram} \]
where the bottom horizontal and right-hand vertical maps are given by $P_2G \to P_1G$ and $P_2F \to P_1F$ respectively.

For $n = 3$, there are three partition types: $3$, $12$ and $111$. This gives a pullback square
\[ \begin{diagram}
  \node{P_3(FG)} \arrow{e} \arrow{s} \node{(P_3F)(P_1G) \times (P_1F)(P_3G) \times (P_1 P_1 \creff^2(F)(P_1G,P_2G))} \arrow{s} \\
  \node{P_2(FG)} \arrow{e} \node{(P_2F)(P_1G) \times (P_1F)(P_2G) \times (P_1 P_1 \creff^2(F)(P_1G,P_1G))}
\end{diagram} \]
We can identify the co-cross-effects in this case as
\[ P_1 P_1 \creff^2(F)(P_1G,P_kG) \homeq \der_2F \smsh P_1G \smsh P_kG \]
The most difficult part of this square is the map
\[ P_2(FG) \to \der_2F \smsh P_1G \smsh P_1G \]
that forms part of the bottom arrow in the above diagram. It can be written as a composite:
\[ P_2(FG) \to P_2F \circ P_1G \to \creff^2(P_2F)(P_1G,P_1G) \homeq \der_2F \smsh P_1G \smsh P_1G. \]
where the second map is essentially $\Delta_2$ (see Definition \ref{def:cross-effect-maps}). Unpacking the square, we can write $P_3(FG)$ as the homotopy limit of the following diagram:
\[ \begin{diagram}
  \node{(P_1F)(P_3G)} \arrow{s} \node{(P_3F)(P_1G)} \arrow{s} \node{\der_2F \smsh P_1G \smsh P_2G} \arrow{s} \\
  \node{(P_1F)(P_2G)} \arrow{se} \node{(P_2F)(P_1G)} \arrow{s} \arrow{e,t}{\Delta_2} \node{\der_2F \smsh P_1G^{\smsh 2}} \\
  \node[2]{(P_1F)(P_1G)}
\end{diagram} \]
\end{examples}

Theoretically, this process could be used to get `formulas' for all $P_n(FG)$ as homotopy limits. In practice, these diagrams seem to quickly get rather complicated and identifying the maps involved depends on having a lot of information about the functor $F$, including models for all the functors $P_{k_1} \dots P_{k_r}\creff^r(F)$ and for the maps (similar to the diagonal maps $\Delta_r$) that go between them.

\section{Fundamental results on Taylor towers of composite functors} \label{sec:fundamental}

The main aim of this section is to prove the following crucial results. They say, effectively, that to study $P_n(FG)$, it is sufficient to know $P_nF$ and $P_nG$.

\begin{proposition} \label{prop:fundamental}
Let $F,G: \spectra \to \spectra$ be simplicial homotopy functors with $G$ reduced. Then:
\begin{enumerate}
  \item the canonical map $P_n(FG) \to P_n((P_nF)G)$ is an equivalence;
  \item if $F$ is finitary, then the canonical map $P_n(FG) \to P_n(F(P_nG))$ is an equivalence.
\end{enumerate}
\end{proposition}

It should be noted that the finitary condition is definitely necessary in part (2) of this proposition. Again Example \ref{ex:counterexample} provides a counterexample if that condition is dropped.

We start with the proof of part (1) of Proposition \ref{prop:fundamental}. For this, we recall the details of Goodwillie's construction of $P_nF$ for a homotopy functor $F$. We quote from Kuhn \cite[\S5]{kuhn:2007} (which generalizes slightly the original definitions of Goodwillie \cite{goodwillie:2003}).

\begin{definition} \label{def:T_n}
Let $X \in \spectra$ and let $U$ be a finite set. We define the \emph{join} of $X$ and $U$ to be the homotopy cofibre of the fold map:
\[ U * X := \hocofib \left(\Wdge_{U} X \to X \right). \]
For a homotopy functor $F:\spectra \to \spectra$, we define a new functor $T_nF: \spectra \to \spectra$ by
\[ T_nF(X) := \holim_{\emptyset \neq U \subset \{0,\dots,n\}} F(U * X). \]
This homotopy limit is taken over the category of nonempty subsets of the set $\{0,\dots,n\}$ whose morphisms are inclusions. There is a natural map
\[ F(X) \to T_nF(X). \]
Goodwillie's $n$-excisive approximation to $F$ is then given by
\[ P_nF(X) := \hocolim(F(X) \to T_nF(X) \to T_n(T_nF)(X) \to \dots). \]
\end{definition}

The key to the proof of part (1) of Proposition \ref{prop:fundamental} is the following construction.

\begin{definition} \label{def:r_n(F,G)}
Suppose $G$ is simplicial and \emph{strictly reduced}, that is $G(*) = *$. We define a map
\[ r_n(F,G) : (T_nF)(GX) \to T_n(FG)(X). \]
According to Definition \ref{def:T_n} this should be a map
\[ \holim_{U} F(U * GX) \to \holim_{U} F(G(U * X)). \]
The definition of $r_n(F,G)$ is then completed by describing a map
\[ U * GX \to G(U * X) \]
that is natural with respect to $X$, $G$ and $U$. Since the join $U * X$ is described as a homotopy cofibre, it is sufficient to construct maps of the form
\[ \hocofib(GX \to GY)) \to G(\hocofib(X \to Y)) \]
or, explicitly,
\[ (GY \wdge_{GX} (I \smsh GX)) \to G(Y \wdge_X (I \smsh X)). \]
where $I$ is the based interval. It is therefore sufficient that there be a natural map
\[ I \smsh GX \to G(I \smsh X). \]
This exists if $G$ is enriched over based simplicial sets, or equivalently, if $G$ is simplicial and $G(*) = *$.
\end{definition}

Since Proposition \ref{prop:fundamental} assumes only that $G$ is (weakly) reduced, i.e. that $G(*) \homeq *$, we note that any reduced functor is equivalent to one that is strictly reduced. The following argument was told to the author by Greg Arone.

\begin{lemma}
Let $G: \spectra \to \spectra$ be a reduced homotopy functor where $\spectra$ denotes the category of symmetric spectra based on simplicial sets as in \cite{hovey/shipley/smith:2000}. Then there is an equivalence $G \weq \tilde{G}$ where $\tilde{G}:\spectra \to \spectra$ has $\tilde{G}(*) = *$.
\end{lemma}
\begin{proof}
For any $X$, the composite $G(*) \to G(X) \to G(*)$ is the identity on $G(*)$. On the level of the simplicial sets that make up these symmetric spectra, this tells us that $G(*)$ can be identified with a levelwise subspectrum of $G(X)$. We then define
\[ \tilde{G}(X) := G(X)/G(*). \]
The natural map $G(X) \to \tilde{G}(X)$ is a levelwise weak equivalence because an inclusion of simplicial sets is a cofibration. It is therefore a stable weak equivalence and clearly $\tilde{G}(*) = *$.
\end{proof}

Now note the following important property of the map $r_n(F,G)$ constructed in Definition \ref{def:r_n(F,G)}.

\begin{lemma} \label{lem:T_n-commute}
The following diagram commutes:
\[ \begin{diagram}
  \node{FGX} \arrow{e,t}{t_n(FG)} \arrow{s,l}{(t_nF)(GX)} \node{T_n(FG)(X)} \\
  \node{(T_nF)(GX)} \arrow{ne,b}{r_n(F,G)}
\end{diagram} \]
\end{lemma}
\begin{proof}
This follows from the fact that the following diagram commutes
\[ \begin{diagram}
  \node{GX} \arrow{e} \arrow{s} \node{G(U * X)} \\
  \node{U * GX} \arrow{ne}
\end{diagram} \]
where the diagonal map is that described in Definition \ref{def:r_n(F,G)}.
\end{proof}

\begin{definition}
Now consider the maps
\[ T_n^k(r_n(F,G)): T_n^k((T_nF)G) \to T_n^{k+1}(FG). \]
Taking the homotopy colimit over increasing $k$, we get a map
\[ u_n(F,G): P_n((T_nF)G) \to P_n(FG). \]
\end{definition}

\begin{lemma} \label{lem:T_n-identity}
The composite
\[ \begin{diagram} \node{P_n(FG)} \arrow{e,t}{t_nF} \node{P_n((T_nF)G)} \arrow{e,t}{u_n(F,G)} \node{P_n(FG)} \end{diagram} \]
is homotopic to the identity.
\end{lemma}
\begin{proof}
This follows from Lemma \ref{lem:T_n-commute}
\end{proof}

\begin{definition}
Composing the maps $u_n(T_n^kF,G)$ we get a map
\[ P_n((T_n^kF)G) \to P_n(FG) \]
that together define a map
\[ v_n(F,G): P_n((P_nF)G) \to P_n(FG) \]
which we will show is an inverse to the canonical map $P_n(FG) \to P_n((P_nF)G)$.
\end{definition}

\begin{lemma} \label{lem:P_n-identity}
The composite
\[ \begin{diagram} \node{P_n(FG)} \arrow{e,t}{p_nF} \node{P_n((P_nF)G)} \arrow{e,t}{v_n(F,G)} \node{P_n(FG)} \end{diagram} \]
is homotopic to the identity.
\end{lemma}
\begin{proof}
This follows from Lemma \ref{lem:T_n-identity}.
\end{proof}

Lemma \ref{lem:P_n-identity} shows that the canonical map
\[ p_nF: P_n(FG) \to P_n((P_nF)G) \]
has a left inverse, but it is far from obvious that the definition we have given for $v_n(F,G)$ is also a right inverse. However, we can use the following trick to complete the proof of part (1) of Proposition \ref{prop:fundamental}.

\begin{proof}[Proof of part (1) of Proposition \ref{prop:fundamental}]
Let $Q_{n+1}F$ denote the homotopy fibre of the map
\[ p_nF: F \to P_nF. \]
Since $P_n$ commutes with homotopy fibre, part (1) of Proposition \ref{prop:fundamental} will follow if we can show that
\[ P_n((Q_{n+1}F)G) \homeq *. \]
Now consider Lemma \ref{lem:P_n-identity} applied to the functors $Q_{n+1}F$ and $G$. We then see that the composite
\[ P_n((Q_{n+1}F)G) \to P_n((P_nQ_{n+1}F)G) \to P_n((Q_{n+1}F)G) \]
is homotopic to the identity.  But the middle term here is trivial since $P_n(Q_{n+1}F) \homeq *$. It then follows that
\[ P_n((Q_{n+1}F)G) \homeq * \]
as required.
\end{proof}

\begin{remarks} \label{rem:part1} \hfill
\begin{enumerate}
  \item Part (1) of Proposition \ref{prop:fundamental} is true for functors from spaces to spaces and our proof carries over directly to that case.
  \item There should be a more general version that is true without the requirement that $G$ be reduced. This would say that
  \[ P_n(FG) \to P_n((P_n^{G(*)}F)G) \]
  is an equivalence, where $P_n^{G(*)}F$ denotes the \ord{n} term in the Taylor tower of $F$ expanded at the point $G(*)$.
\end{enumerate}
\end{remarks}

We now turn to part (2) of Proposition \ref{prop:fundamental}. Our proof of this will have a distinctly different flavour than that of part (1). A proof along similar lines is probably possible, and we could construct at least a one-sided inverse to the map $P_n(FG) \to P_n(F(P_nG))$ in much the same way. However, the trick we used in the proof of part (1) will not work since $F(Q_{n+1}G)$ is not necessarily equivalent to the homotopy fibre of the map $FG \to F(P_nG)$.

Instead, we give a proof by induction on the Taylor tower of $F$. The main part of this induction will be showing that the result is true when $F$ is homogeneous, so we start by proving this.

\begin{lemma} \label{lem:finitarycase}
Let $F: \spectra \to \spectra$ be a finitary $k$-homogeneous functor and let $G: \spectra \to \spectra$ be any homotopy functor. Then the canonical map
\[ P_n(FG) \to P_n(F(P_nG)) \]
is an equivalence.
\end{lemma}
\begin{proof}
Since $F$ is $k$-homogeneous, we know that
\[ F(X) \homeq L(X,\dots,X)_{h\Sigma_k} \]
for some symmetric multilinear functor $L: \spectra^k \to \spectra$. Since $P_n$ commutes with homotopy orbits for spectrum-valued functors, it is then sufficient to show that the map
\[ L(G,\dots,G) \to L(P_nG,\dots,P_nG) \]
is an equivalence after applying $P_n$. To see this, we write this map as a composite
\[ \begin{split} L(G,\dots,G) &\to L(P_nG,G,\dots,G) \\
                            &\to \dots \\
                            &\to L(P_nG,\dots,P_nG) \end{split} \]
of maps where we change one of the factors from $G(X)$ to $P_nG(X)$ at a time. It is then enough to show that each of these maps is a $P_n$-equivalence.

The homotopy fibre of one of these maps looks like
\[ \tag{*} L(P_nG,\dots,P_nG,Q_{n+1}G,G,\dots,G) \]
where $Q_{n+1}G = \hofib(G \to P_nG)$. This is because linear functors preserve fibre sequences. It is now sufficient to show that $P_n$ of this is trivial.

But now consider the functor
\[ (X_1,\dots,X_k) \mapsto L(P_nG(X_1),\dots,P_nG(X_{i-1}), Q_{n+1}G(X_i), G(X_{i+1}),\dots,G(X_k)). \]
Since $Q_{n+1}G$ is $(n+1)$-reduced, this functor is at least $(0,\dots,0,n+1,0,\dots,0)$-reduced, by Lemma \ref{lem:reduced}. But then by \cite[6.10]{goodwillie:2003}, it follows that (*) above is at least $n+1$-reduced, which is precisely what we wanted to show.
\end{proof}

We can now carry out the induction and complete the proof of part (2) of Proposition \ref{prop:fundamental}.

\begin{proof}[Proof of part (2) of Proposition \ref{prop:fundamental}]
First consider the following diagram:
\[ \begin{diagram}
  \node{P_n(FG)} \arrow{e,t}{\sim} \arrow{s} \node{P_n((P_nF)G)} \arrow{s} \\
  \node{P_n(F(P_nG))} \arrow{e,t}{\sim} \node{P_n((P_nF)(P_nG))}
\end{diagram} \]
The horizontal maps are equivalences by part (1) of Proposition \ref{prop:fundamental}. This reduces to the case where $F$ is $n$-excisive.

Now consider the diagram
\[ \begin{diagram}
  \node{(D_kF)G} \arrow{e} \arrow{s} \node{(D_kF)P_nG} \arrow{s} \\
  \node{(P_kF)G} \arrow{e} \arrow{s} \node{(P_kF)P_nG} \arrow{s} \\
  \node{(P_{k-1})G} \arrow{e} \node{(P_{k-1}F)P_nG}
\end{diagram} \]
The vertical maps form fibre sequences. The top horizontal map is an equivalence by Lemma \ref{lem:finitarycase}. By induction, we can assume that the bottom horizontal map is an equivalence. It follows that the middle horizontal map is an equivalence. The base case of the induction concerns the map
\[ P_n((P_0F)G) \to P_n((P_0F)(P_nG)) \]
which is an equivalence because each side is just the constant functor with value $F(*)$. This completes the proof.
\end{proof}

\begin{remarks} \hfill
\begin{enumerate}
  \item Our proof of part (2) of Proposition \ref{prop:fundamental} does not carry over directly to the case of functors of based spaces. However, the proposition is still true in that context. To prove this, we use the characterization of finitary homogeneous functors of based spaces as of the form
      \[ X \mapsto \Omega^\infty(E \smsh (\Sigma^\infty X)^{\smsh k})_{h\Sigma_k}. \]
      To complete the proof, we then need to prove the additional claim that
      \[ P_n(\Sigma^\infty G) \to P_n(\Sigma^\infty P_nG) \]
      is an equivalence (i.e. we have reduced to the case $F = \Sigma^\infty$). This is not immediately obvious but follows by showing that the homotopy \emph{cofibre} of $G \to P_nG$ is $n+1$-reduced.
  \item The condition that $G$ is reduced should not be necessary here. The only place it is used is when we apply part (1) of Proposition \ref{prop:fundamental} to reduce to the case that $F$ is $n$-excisive. As we suggested in \ref{rem:part1}(2), there should be a corresponding version of that result in the case that $G$ is not reduced that is based on the Taylor tower for $F$ expanded at $G(*)$. With such a result, we would prove part (2) of Proposition \ref{prop:fundamental} by carrying out the induction on that Taylor tower instead.
\end{enumerate}
\end{remarks}

% Bibliography (mcching.bib)
\bibliographystyle{amsplain}
\bibliography{mcching}

\end{document}